\documentclass{amsart}
\usepackage[toc,page]{appendix}
\usepackage[headheight=12.0pt]{geometry}
\usepackage[all,cmtip]{xy}
\usepackage{amsmath}
\usepackage{amsfonts}
\usepackage{amssymb}
\usepackage{mathrsfs}   
\usepackage{amsthm}
\usepackage{xcolor}
\usepackage{cite}
\usepackage{stmaryrd}
\usepackage{graphicx}
\usepackage{pgf}
\usepackage{eepic}
\usepackage{pstricks}
\usepackage{epsfig}
\usepackage{fancyhdr}
\usepackage{tikz,caption,float}
\usepackage[backref=page]{hyperref}

\usepackage{amsbsy}

\hypersetup{
     colorlinks   = true,
     citecolor    = red,
     linkcolor = blue,
     urlcolor = purple
}
    
\usepackage{mathptmx}  
  \newcommand\imCMsym[4][\mathord]{%
  \DeclareFontFamily{U} {#2}{}
  \DeclareFontShape{U}{#2}{m}{n}{
    <-6> #25
    <6-7> #26
    <7-8> #27
    <8-9> #28
    <9-10> #29
    <10-12> #210
    <12-> #212}{}
  \DeclareSymbolFont{CM#2} {U} {#2}{m}{n}
  \DeclareMathSymbol{#4}{#1}{CM#2}{#3}
}
\newcommand\alsoimCMsym[4][\mathord]{\DeclareMathSymbol{#4}{#1}{CM#2}{#3}}

\imCMsym{cmmi}{124}{\CMjmath}
\imCMsym[\mathop]{cmsy}{113}{\CMamalg}
\imCMsym[\mathop]{cmex}{96}{\CMcoprod}
\alsoimCMsym[\mathop]{cmex}{97}{\CMbigcoprod}

\theoremstyle{plain}
\newtheorem*{theoremu}{Theorem}
\newtheorem{theorem}{Theorem}[section]

\newtheorem{proposition}[theorem]{Proposition}

\newtheorem{corollary}[theorem]{Corollary}

\newtheorem{lemma}[theorem]{Lemma}

\theoremstyle{definition}

\newtheorem{definition}[theorem]{Definition}

\theoremstyle{remark}
\newtheorem{remark}[theorem]{Remark}

\newcommand{\N}{{\mathbb N}}
\newcommand{\Z}{{\mathbb Z}}
\newcommand{\Q}{{\mathbb Q}}

\newcommand{\C}{{\mathbb C}}

\newcommand{\F}{{\mathbb F}}

\newcommand{\isomto}{\overset{\sim}{\rightarrow}}

\newcommand{\bu}{\bullet}

\newcommand{\lser}[1]{(\!(#1)\!)}
\newcommand{\pow}[1]{\llbracket #1 \rrbracket}

\newcommand{\tate}[1]{\langle #1 \rangle}

\newcommand{\spec}[1]{\mathrm{Spec}\left(#1\right)}

\newcommand{\cur}[1]{\mathcal{#1}}

\newcommand{\pn}{(\varphi,\nabla)}

\newcommand{\cris}{\mathrm{cris}}
\newcommand{\rig}{\mathrm{rig}}

\newcommand{\et}{\mathrm{\acute{e}t}}

\title{Cycle classes in overconvergent rigid cohomology and a semistable Lefschetz $\bm{(1,1)}$ theorem}

\author{Christopher Lazda}
       \address{ Korteweg-de\thinspace Vries Institute\\ Universiteit van Amsterdam\\ P.O. Box 94248 \\1090 GE\\ Amsterdam\\ the Netherlands}
       \email{c.d.lazda@uva.nl}
       
\author{Ambrus P\'al}
       \address{ Department of Mathematics\\ Imperial College London \\
       Huxley Building, 180 Queen's Gate\\
       South Kensington, London, SW7 2AZ\\
       UK}
       \email{a.pal@imperial.ac.uk}       

\SetSymbolFont{stmry}{bold}{U}{stmry}{m}{n}
\usepackage{bm}

\begin{document}

\begin{abstract} In this article we prove a semistable version of the variational Tate conjecture for divisors in crystalline cohomology, showing that for $k$ a perfect field of characteristic $p$, a rational (logarithmic) line bundle on the special fibre of a semistable scheme over $k\pow{t}$ lifts to the total space if and only if its first Chern class does. The proof is elementary, using standard properties of the logarithmic de\thinspace Rham--Witt complex. As a corollary, we deduce similar algebraicity lifting results for cohomology classes on varieties over global function fields. Finally, we give a counter example to show that the variational Tate conjecture for divisors cannot hold with $\Q_p$-coefficients.
\end{abstract}

\maketitle 

\tableofcontents

\section*{Introduction}

Many of the deepest conjectures in arithmetic and algebraic geometry concern the existence of algebraic cycles on varieties with certain properties. For example, the Hodge and Tate conjectures state, roughly speaking, that on smooth and projective varieties over $\C$ (Hodge) or finitely generated fields (Tate) every cohomology class which `looks like' the class of a cycle is indeed so. One can also pose variational forms of these conjectures, giving conditions for extending algebraic classes from one fibre of a smooth, projective morphism $f:X\rightarrow S$ to the whole space. For divisors, the Hodge forms of both these conjectures (otherwise known as the Lefschetz $(1,1)$ theorem) are relatively straightforward to prove, using the exponential map, but even for divisors the Tate conjecture remains wide open in general.

Applying the principle that deformation problems in characteristic $p$ should be studied using $p$-adic cohomology, Morrow in \cite{Mor14} formulated a crystalline variational Tate conjecture for smooth and proper families $f:X\rightarrow S$ of varieties in characteristic $p$, and proved the conjecture for divisors, at least when $f$ is projective. The key step of the proof is a version of this result over $S=\spec{k\pow{t_1,\ldots,t_n}}$, with $k$ a perfect field of characteristic $p$. When $n=1$ this is a direct equicharacteristic analogue of Berthelot and Ogus' theorem \cite[Theorem 3.8]{BO83} on lifting line bundles from characteristic $p$ to characteristic $0$. 

Morrow's proof of the local statement uses some fairly heavy machinery from motivic homotopy theory, in particular a `continuity' result for topological cyclic homology. In this article we provide a new proof of the local crystalline variational Tate conjecture for divisors, at least over the base $S=\spec{k\pow{t}}$, which only uses some fairly basic properties of the de\thinspace Rham--Witt complex, and is close in spirit to the approach taken in \cite{Mor15}. The point of giving this proof is that it adapts essentially verbatim to the case of semistable reduction, once the corresponding basic properties of the \emph{logarithmic} de\thinspace Rham--Witt complex are in place.

So let $\cur X$ be a semistable, projective scheme over $k\pow{t}$, with special fibre $X_0$ and generic fibre $X$. Write $K=W(k)[1/p]$ and let $\mathcal{R}$ denote the Robba ring over $K$. Then there is an isomorphism
\[ H^2_\rig(X/\cur{R})^{\nabla=0}\cong H^2_\mathrm{log\text{-}\mathrm{cris}}(X_0^\times/K^\times)^{N=0}\]
between the horizontal sections of the Robba ring-valued rigid cohomology of $X$ and the part of the log-crystalline cohomology of $X_0$ killed by the monodromy operator. The former is defined to be the base change $H^2_\rig(X/\cur{E}^\dagger)\otimes_{\cur{E}^\dagger} \cur{R}$ to the Robba ring of the $\mathcal{E}^\dagger$-valued rigid cohomology $H^2_\rig(X/\cur{E}^\dagger)$ constructed in \cite{LP16}. These groups are $\pn$-modules over $\cur R$ and $\cur E^\dagger$ respectively. In particular, if $\cur L$ is a line bundle on $X_0$, we can view its first Chern class $c_1(\cur L)$ as an element of $H^2_\rig(X/\cur{R})$. Our main result is then the following semistable version of the local crystalline variational Tate conjecture for divisors.

\begin{theoremu}[\ref{mainss}] $\cur L$ lifts to $\mathrm{Pic}(\cur X)_{\Q}$ if and only if $c_1(\cur L)$ lies in $H^2_\rig(X/\cur{E}^\dagger)\subset H^2_\rig(X/\cur{R})$.
\end{theoremu}

There is also a version for logarithmic line bundles on $X_0$. The general philosophy of $p$-adic cohomology over $k\lser{t}$ is that the $\cur{E}^\dagger$-structure $H^i_\rig(X/\cur{E}^\dagger)\subset H^i_\rig(X/\cur{R})$ is the equicharactersitic analogue of the Hogde filtration on the $p$-adic cohomology of varieties over mixed characteristic local fields. With this in mind, this is the direct analogue of Yamashita's semistable Lefschetz (1,1) theorem \cite{Yam11}. As a corollary, we can deduce a global result on algebraicity of cohomology class as follow. Let $F$ be a function field of transcendence degree 1 over $k$, and $X/F$ a smooth projective variety. Let $v$ be a place of semistable reduction for $X$, with reduction $X_v$. In this situation, we can consider the rigid cohomology of $X/K$ (see \S\ref{glob}), and there is a map
\[ \mathrm{sp}_v:\cur{H}^2_\rig(X/K)^{\nabla=0}\rightarrow H^2_\mathrm{log\text{-}\mathrm{cris}}(X_v^\times/K_v^\times)\]
from the second cohomology of $X$ to the log crystalline cohomology of $X_v$. 

\begin{theoremu}[\ref{glma}]  A class $\alpha\in \cur{H}^2_\rig(X/K)^{\nabla=0}$ is in the image of $\mathrm{Pic}(X)_{\Q}$ under the Chern class map if and only if $\mathrm{sp}_v(\alpha)$ is in the image of $\mathrm{Pic}(X_v)_{\Q}$.
\end{theoremu}

One might wonder whether the analogue of the crystalline variational Tate conjecture holds for line bundles with $\Q_p$-coefficients (in either the smooth or semistable case). Unfortunately, the answer is no. Indeed, if it were true, then it follows relatively easily that the analogue of Tate's isogeny theorem would hold over $k\lser{t}$, in other words for any two abelian varieties $A,B$ over $k\lser{t}$, the map
\[ \mathrm{Hom}(A,B)\otimes \Q_p \rightarrow \mathrm{Hom}(A[p^\infty],B[p^\infty])\otimes_{\Z_p} \Q_p  \]
would be an isomorphism. That this cannot be true is well-known, and examples can be easily provided with both $A$ and $B$ elliptic curves.

Let us now summarise the contents of this article. In \S1 we show that the cycle class map in rigid cohomology over $k\lser{t}$ descends to the bounded Robba ring. In \S2 we recall the relative logarithmic de\thinspace Rham--Witt complex, and prove certain basic properties of it that we will need later on. In \S3 we reprove a special case of the key step in Morrow's article \cite{Mor14}, showing the crystalline variational Tate conjecture for smooth and projective schemes over $k\pow{t}$. The argument we give is elementary. In \S4 we prove the semistable version of the crystalline variational Tate conjecture over $k\pow{t}$, more or less copying word for word the argument in \S3. In \S5 we translate these results into algebraicity lifting results for varieties over \emph{global} function fields. Finally, in \S6 we give a counter-example to the analogue of the of crystalline variational Tate conjecture for line bundles with $\Q_p$-coefficients.

\subsection*{Acknowledgements}

A. P\'al was partially supported by the EPSRC grant P36794. C. Lazda was supported by a Marie Curie fellowship of the Istituto Nazionale di Alta Matematica ``F. Severi''. Both authors would like to thank Imperial College London and the Universit\`a Degli Studi di Padova for hospitality during the writing of this article.

\subsection*{Notations and convenions}

Throughout we will let $k$ be a perfect field of characteristic $p>0$, $W$ its ring of Witt vectors and $K=W[1/p]$. In general we will let $F=k\lser{t}$ be the field of Laurent series over $k$, and $R=k\pow{t}$ its ring of integers (although this will not be the case in \S\ref{glob}). We will denote by $\cur{E}^\dagger,\cur{R},\cur{E}$ respectively the bounded Robba ring, the Robba ring, and the Amice ring over $K$, and we will also write $\cur{E}^+=W\pow{t}\otimes_W K$. For any of the rings $\cur{E}^+$, $\cur{E}^\dagger$, $\cur{R}$, $\cur{E}$ we will denote by $\underline{\mathbf{M}\Phi}^\nabla_{(-)}$ the corresponding category of $\pn$-modules, i.e. finite free modules with connection and horizontal Frobenius. A variety over a given Noetherian base scheme will always mean a separated scheme of finite type. For any abelian group $A$ and any ring $S$ we will let $A_S$ denote $A\otimes_{\Z} S$.

\section{Cycle class maps in overconvergent rigid cohomology}\label{CCM}

Recall that for varieties $X/F$ over the field of Laurent series $F=k\lser{t}$ the rigid cohomology groups $H^i_\rig(X/\cur{E})$ are naturally $\pn$-modules over the Amice ring $\cur{E}$. In the book \cite{LP16} we showed how to canonically descend these cohomology groups to obtain `overconvergent' $\pn$-modules $H^i_\rig(X/\cur{E}^\dagger)$ over the bounded Robba ring $\cur{E}^\dagger$, these groups satisfy all the expected properties of an `extended' Weil cohomology theory. In particular, there exist versions $H^i_{c,\rig}(X/\cur{E})$, $H^i_{c,\rig}(X/\cur{E}^\dagger)$ with compact support.

\begin{definition} Define the (overconvergent) rigid homology of a variety $X/F$ by 
\[ H^\rig_i(X/\cur{E}):=H^i_\rig(X/\cur{E})^\vee,\;\; H^\rig_i(X/\cur{E}^\dagger):=H^i_\rig(X/\cur{E}^\dagger)^\vee  \]
and the (overconvergent) Borel--Moore homology by
\[ H^{\mathrm{BM},\rig}_i(X/\cur{E}):=H^i_{c,\rig}(X/\cur{E})^\vee,\;\;H^{\mathrm{BM},\rig}_i(X/\cur{E}^\dagger):=H^i_{c,\rig}(X/\cur{E}^\dagger)^\vee.\]
\end{definition}

In \cite{Pet03} the author constructs cycle class maps in rigid cohomology, which can be viewed as homomorphisms
\[  A_d(X)\rightarrow H_{2d}^{\mathrm{BM},\rig}(X/\cur{E})\]
from the group of $d$-dimensional cycles modulo rational equivalence. Our goal in this section is the following entirely straightforward result.

\begin{proposition} The cycle class map descends to a homomorphism
\[ A_d(X)\rightarrow H_{2d}^{\mathrm{BM},\rig}(X/\cur{E}^\dagger)^{\nabla=0,\varphi=p^d}. \]\end{proposition} 

\begin{proof} Note that since $H_{2d}^{\mathrm{BM},\rig}(X/\cur{E}^\dagger)^{\nabla=0,\varphi=p^d}\subset H_{2d}^{\mathrm{BM},\rig}(X/\cur{E})$ it suffices to show that for every integral closed subscheme $Z\subset X$ of dimension $d$, the cycle class $\eta(Z)\in H_{2d}^{\mathrm{BM},\rig}(X/\cur{E})$ actually lies in the subspace $H_{2d}^{\mathrm{BM},\rig}(X/\cur{E}^\dagger)^{\nabla=0,\varphi=p^d}$. 

By construction, $\eta(Z)$ is the image of the fundamental class of $Z$ (i.e. the trace map $\mathrm{Tr}_Z: H^{2d}_{c,\rig}(Z/\cur{E})\rightarrow \cur{E}(-d)$) under the map
\[ H_{2d}^{\mathrm{BM},\rig}(Z/\cur{E}) \rightarrow H_{2d}^{\mathrm{BM},\rig}(X/\cur{E}) \]
induced by the natural map $H^{2d}_{c,\rig}(X/\cur{E})\rightarrow H^{2d}_{c,\rig}(Z/\cur{E})$ in compactly supported cohomology. Hence it suffices to simply observe that both this map and the trace map descend to horizontal, Frobenius equivariant maps on the level of $\cur{E}^\dagger$-valued cohomology. Alternatively, we could observe that both $H^{2d}_{c,\rig}(X/\cur{E})\rightarrow H^{2d}_{c,\rig}(Z/\cur{E})$ and $\mathrm{Tr}_Z$ are horizontal and Frobenius equivariant at the level of $\cur{E}$-valued cohomology, which gives
\[  A_d(X)\rightarrow H_{2d}^{\mathrm{BM},\rig}(X/\cur{E})^{\nabla=0,\varphi=p^d}, \]
then applying Kedlaya's full faithfulness theorem \cite[Theorem 5.1]{Ked04c} gives an isomorphism 
\[ H_{2d}^{\mathrm{BM},\rig}(X/\cur{E})^{\nabla=0,\varphi=p^d}\cong  H_{2d}^{\mathrm{BM},\rig}(X/\cur{E}^\dagger)^{\nabla=0,\varphi=p^d}.\]
\end{proof}

\section{Preliminaries on the de Rham--Witt complex}

The purpose of this section is to gather together some results we will need on the various de\thinspace Rham--Witt complexes that will be used throughout the article. These are all generalisations to the logarithmic case of well-known results from \cite{Ill79}, and should therefore present no surprises. The reader will not lose too much by skimming this section on first reading and referring back to the results as needed.

We will, as throughout, fix a perfect ground field $k$ of characteristic $p>0$, all (log)-schemes will be considered over $k$. Given a morphism $(Y,N)\rightarrow (S,L)$ of fine log schemes over $k$, Matsuue in \cite{Mat17} constructed a relative logarithmic de\thinspace Rham--Witt complex $W_\bu\omega^*_{(Y,N)/(S,L)}$, denoted $W_\bu\Lambda^*_{(Y,N)/(S,L)}$ in \cite{Mat17}. This is an \'etale sheaf on $Y$ equipped with operators $F,V$ satisfying all the usual relations (see for example \cite[Definition 3.4(v)]{Mat17}) and which specialises to various previous constructions in particular cases.
\begin{enumerate} \item When $S=\spec{k}$ and the log structures $L$ and $N$ are trivial, then this gives the (canonical extension of the) classical de\thinspace Rham--Witt complex $W_\bu\Omega^*_{Y}$ (to an \'etale sheaf on $Y$).
\item More generally, when the morphism $(Y,N)\rightarrow (S,L)$ is strict, it recovers the relative de\thinspace Rham--Witt complex $W_\bu\Omega^*_{Y/S}$ of Langer and Zink \cite{LZ04}.
\item When the base $(S,L)$ is the scheme $\spec{k}$ with the log structure of the punctured point, and $(Y,N)$ is of semistable type (i.e. \'etale locally \'etale over $k[x_1,\ldots,x_{d+1}]/(x_1\cdots x_c)$ with the canonical log structure) then Matsuue's complex is isomorphic the logarithmic de Rham--Witt complex $W\omega^*_{Y}$ studied in \cite{HK94}.
\item If we take $(Y,N)$ semistable but instead equip $\spec{k}$ with the trivial log structure, the resulting complex is isomorphic to the one denoted $W\tilde{\omega}^*_Y$ in \cite{HK94}.
\end{enumerate}

If we are given a morphism of log schemes $(Y,N)\rightarrow (S,L)$ over $k$, then as in \cite[\S2.2]{Mat17} we can lift the log structure $N\rightarrow \cur{O}_Y$ to a log structure $W_rN\rightarrow W_r\cur{O}_Y$, where by definition $W_rN=N\oplus \ker \left( (W_r\cur{O}_Y)^* \rightarrow \cur{O}_Y^* \right)$ and the map $N\rightarrow W_r\cur{O}_Y$ is the Techm\"uller lift of $N\rightarrow \cur{O}_Y$. Since $W_r\omega^1_{(Y,N)/(S,L)}$ is a quotient of the pd-log de\thinspace Rham complex $\breve{\omega}^*_{(W_rY,W_rN)/(W_rS,W_rL)}$ (see \cite[\S3.4]{Mat17}) there is a natural map $d\log : W_rN \rightarrow W_r\omega^1_{(Y,N)/(S,L)}$ and hence we obtain maps
\[ d\log : N^\mathrm{gp} \rightarrow W_r\omega^1_{(Y,N)/(S,L)} \]
which are compatible as $r$ varies. We let $W_r\omega^1_{(Y,N)/(S,L),\log}$ denote the image.

When both log structures are trivial, and $Y\rightarrow \spec{k}$ is smooth, then \cite[Proposition I.3.23.2]{Ill79} says that $d\log$ induces an exact sequence
\[ 1 \rightarrow \left(\cur{O}_{Y}^*\right)^{p^r} \rightarrow \cur{O}_Y^*\rightarrow W_r\Omega^1_{Y,\log}\rightarrow 0,\]
and our first task in this section to obtain an analogue of this result for semistable log schemes over $k$. In fact, since we will really only be interested in the case when $Y$ arises as the special fibre of a semistable scheme over $k\pow{t}$, we will only treat this special case.

We will therefore let $\cur X$ denote a semistable scheme over $R=k\pow{t}$ (not necessarily proper). We will let $L$ denote the log structure given by the closed point of $\spec{R}$, and write $R^\times=(R,L)$. We will denote by $L_n$ the inverse image log structure on $R_n=k\pow{t}/(t^{n+1})$, and write $R_n^\times=(R_n,L_n)$. We will also write $k^\times=(k,L_0)$. We will denote by $M$ the log structure on $\cur X$ given by the special fibre, and write $\cur{X}^\times=(\cur{X},M)$. Similarly we have log structures $M_n$ on $X_n=\cur{X} \otimes_R R_n$, and we will write $X_n^\times=(X_n,M_n)$. Finally, when considering the logarithmic de\thinspace Rham--Witt complex relative to $k$ (with the trivial log structure) we will drop $k$ from the notation, e.g. we will write $W_r\omega^*_{X_0^\times}$ instead of $W_r\omega^*_{X_0^\times/k}$.

\begin{proposition} \label{fesl1} The sequence
\[ 0\rightarrow p^rM_0^\mathrm{gp} \rightarrow M_0^\mathrm{gp}\overset{d\log}{\longrightarrow} W_r\omega^1_{X_0^\times,\log} \rightarrow 0 \]
is exact.
\end{proposition}

\begin{proof} The surjectivity of the right hand map and the injectivity of the left hand map are by definition, and since $p^rW_r\omega^1_{X_0^\times,\log}=0$, the sequence is clearly a complex. The key point is then to show exactness in the middle. So suppose that we are given $m\in M_0^\mathrm{gp}$ is such that $d\log m=0$. We will show that $m\in p^rM_0^\mathrm{gp}$ by induction on $r$.

When $r=1$ we note that the claim is \'etale local, we may therefore assume $X_0^\times$ to be affine, \'etale and strict over $\spec{\frac{k[x_1,\ldots,x_d]}{(x_1\cdots x_c)}}$, say $X_0=\spec{A}$. We have
\[ \omega^1_{(A,\N^c)} \cong \bigoplus_{i=1}^c A\cdot d\log x_i \oplus \bigoplus_{i=c+1}^{d}A\cdot dx_i.\]
Now suppose that we are given a local section $m=u\prod_{i=1}^c x_i^{n_i}$ of $M_0^\mathrm{gp}$ for $u\in A^*$ and $n_i\in\Z$. Write
\[d\log u = \sum_{i=1}^c a_i d\log x_i + \sum_{i=c+1}^d a_i dx_i \]
with $a_i\in A$, note that since $d\log u$ actually comes from an element of $\Omega^1_{A}$ it follows that $a_i\in x_iA$ for $1\leq i\leq c$. In particular, we have $n_i=-x_ib_i$ for $1\leq i\leq c$ and some $b_i\in A$; passing to $A/x_iA$ it therefore follows that $n_i=0$ in $k$. Hence each $n_i$ is divisible by $p$. It follows that $\prod_{i=1}^{c}x_i^{n_i}$ is in $pM_0^\mathrm{gp}$, and its $d\log$ vanishes. By dividing by this element we may therefore assume that $m=u\in A^*$. Since semistable schemes are of Cartier type, we may apply \cite[Theorem 4.12]{Kat89}, which tells us that (\'etale locally) $u\in A^{(p)*}$ (since $d\log u=0\Rightarrow du=0$). Since $k$ is perfect, $A^{(p)*}=(A^*)^p$ and we may conclude.

When $r>1$ and $d\log m=0\in W_r\omega^1_{X_0^\times,\log}$, then in particular $d\log m=0 \in W_{r-1}\omega^1_{X_0^\times}$; hence by applying the induction hypothesis we obtain $m=p^{r-1}m_1$. But now this implies that $p^{r-1}d\log m_1=0\in W_r\omega^1_{X_0^\times}$, we claim that in fact it follows that $d\log m_1=0\in \omega^1_{X_0^\times}$. Indeed, since $\omega^1_{X_0^\times}$ is a locally free $\cur{O}_{X_0}$-module, to prove that a section vanishes it suffices to show that it does so on a dense open subscheme. In particular, by restricting to the smooth locus of $X_0$ we can assume that $X_0$ is smooth and the log structure is given by $\cur{O}_{X_0}^*\oplus \N$, $(u,n)\mapsto u.0^n$. We now apply \cite[Proposition I.3.4]{Ill79} and \cite[Lemma 7.4]{Mat17} to conclude that $d\log m_1=0$ as required. Thus applying the case $r=1$ finishes the proof.
\end{proof}

The following is analogous to \cite[Corollaire I.3.27]{Ill79}.

\begin{proposition} \label{omfe}The sequences of pro-sheaves
\begin{align*} 0 \rightarrow \left\{ W_r\omega^1_{\cur{X}^{\times},\log} \right\}_r \rightarrow &\left\{ W_r\omega^1_{\cur{X}^{\times}} \right\}_r \overset{1-F}{\rightarrow} \left\{ W_r\omega^1_{\cur{X}^{\times}} \right\}_r \rightarrow 0 ,\\
0 \rightarrow \left\{ W_r\omega^1_{X_0^\times/k^\times,\log} \right\}_r \rightarrow &\left\{ W_r\omega^1_{X_0^\times/k^\times} \right\}_r \overset{1-F}{\rightarrow} \left\{ W_r\omega^1_{X_0^\times/k^\times} \right\}_r \rightarrow 0
\end{align*}
are exact.
\end{proposition}

\begin{proof} Let us consider the first sequence. We will apply N\'eron--Popescu desingularisation \cite[Theorem 1.8]{Pop86} to write $\mathcal{X}$ as a cofiltered limit $\mathcal{X}=\varprojlim_{\alpha\in A} X^\alpha$ of schemes $X^\alpha$ which are smooth over $k$. Moreover, after possibly changing the indexing category $A$ we may assume that there exist closed subschemes $D^\alpha \subset X^\alpha$ such that:
\begin{itemize}
\item $D^\beta= D^\alpha\times_{X^\alpha} X^\beta$ for all $\beta\rightarrow \alpha$,
\item $X_0=D^\alpha\times_{X^\alpha}\mathcal{X}$ for all $\alpha$.
\end{itemize}
Again, after possible changing the index category $A$ we may assume that each $D^\alpha\subset X^\alpha$ is a normal crossings divisor. Thus using the fact that the logarithmic de\thinspace Rham--Witt complex commutes with filtered colimits, we may reduce to considering the analogous question for $Y$ smooth over $k$ with log structure $N$ coming from a normal crossings divisor $D\subset Y$. The claim is \'etale local, we may therefore assume that $Y$ is \'etale over $k[x_1,\ldots,x_n]$ with $D$ the inverse image of $\{ x_1\cdots x_c=0 \} $. Locally, $N$ is generated by $\cur{O}_{Y}^*$ and $x_i$ for $1\leq i\leq c$, so in order to see that the sequence is a complex, or in other words that $(1-F)(d\log n)=0$, it suffices to check that $(1-F)(d\log x_i)=0$. This is a straightforward calculation. For the surjectivity of $1-F$ we claim in fact that
\[ 1-F: W_{r+1}\omega^1_{(Y,N)} \rightarrow W_{r}\omega^1_{(Y,N)}\]
is surjective. For this we note that by \cite[\S9]{Mat17} there exists an exact sequence
\[ 0 \rightarrow W_r\Omega^1_Y \rightarrow W_r\omega^1_{(Y,N)} \rightarrow \bigoplus_{i=1}^c W_r\cur{O}_{D_i}\cdot d\log x_i \rightarrow 0\]
for all $r$, where $D_i$ are the irreducible components of $D$. Denote the induced map $W_r\omega^1_{(Y,N)} \rightarrow  W_r\cur{O}_{D_i}$ by $\mathrm{Res}_i$. (In fact, it is assumed in \cite[\S9]{Mat17} that $Y$ is proper over $k$, however, the \emph{proof} of the exactness in \cite[\S8.2]{Mat17} is local, and therefore works equally well in the non-proper case.) Since $(1-F)(d\log x_i)=0$ it follows that we have the commutative diagram
\[ \xymatrix{ 0 \ar[r] & W_{r+1}\Omega^1_Y \ar[r]\ar[d]_{1-F} & W_{r+1} \omega^1_{(Y,N)} \ar[r]\ar[d]_{1-F} & \bigoplus_{i=1}^c W_{r+1} \cur{O}_{D_i} \ar[r]\ar[d]_{1-F} & 0 \\ 0 \ar[r] & W_{r}\Omega^1_Y \ar[r] & W_{r} \omega^1_{(Y,N)} \ar[r] & \bigoplus_{i=1}^c W_{r} \cur{O}_{D_i} \ar[r] & 0 } \]
where $W_r\Omega^1_Y$ is the usual (non-logarithmic) de\thinspace Rham--Witt complex of $Y$. It therefore suffices to apply \cite[Propositions I.3.26, I.3.28]{Ill79}, stating that the left and right vertical maps are surjective. Finally, to show exactness in the middle, suppose that we are given $\omega\in W_{r+1} \omega^1_{(Y,N)}$ such that $(1-F)(\omega)=0$. Then applying \cite[Proposition I.3.28]{Ill79} we can see that
\[ \mathrm{Res}_i(\omega)\in \Z/p^{r+1}\Z + \ker\left(W_{r+1} \cur{O}_{D_i}\rightarrow  W_{r} \cur{O}_{D_i} \right) \]
for all $i$. Hence after subtracting off an element of $d\log(N^\mathrm{gp})$ we may assume that in fact
\[ \omega\in W_{r+1}\Omega^1_Y+ \ker\left( W_{r+1} \omega^1_{(Y,N)}\rightarrow W_{r} \omega^1_{(Y,N)}\right) .\]
Now applying \cite[Corollaire I.3.27]{Ill79} tells us that 
\[ \omega\in d\log(N^\mathrm{gp}) +\ker\left( W_{r+1} \omega^1_{(Y,N)}\rightarrow W_{r} \omega^1_{(Y,N)}\right) \]
and hence the given sequence of pro-sheaves is exact in the middle.

For the second sequence, the surjectivity of $1-F$ follows from the corresponding claim for the first sequence, since sections of $W_r\omega^1_{X_0^\times/k^\times}$ can be lifted locally to $W_r\omega^1_{\cur{X}^\times}$. We may also argue \'etale locally; assuming that $X_0^\times$ is \'etale and strict over $ \spec{\N^c\rightarrow \frac{k[x_1,\ldots,x_d]}{(x_1\cdots x_c)}}$. The fact that the claimed sequence is a complex follows again from observing that $(1-F)(d\log x_i)=0$ for $1\leq i\leq c$. To see exactness in the middle we use the fact that (again working \'etale locally) we have an exact sequence
\[ 0 \rightarrow \bigoplus_i W_r\Omega^1_{D_i} \rightarrow  W_r\omega^1_{X_0^\times/k^\times} \rightarrow \bigoplus_{ij} W_r\cur{O}_{D_{ij}}\rightarrow  0 \]
by \cite[Lemma 8.4]{Mat17}, where $D_i$ are the irreducible components of $X_0^\times$ and $D_{ij}$ their intersections. Moreover, this fits into a diagram
\[ \xymatrix{  0 \ar[r] & \bigoplus_i \cur{O}_{D_i}^* \ar[r] \ar[d]_{d\log} & M_0^\mathrm{gp} \ar[r]\ar[d]_{d\log} & \bigoplus_{ij} \underline{\Z}_{D_{ij}} \ar[r]\ar[d] & 0 \\
0 \ar[r] &  \bigoplus_i W_r\Omega^1_{D_i} \ar[r]\ar[d]_{1-F} & W_r\omega^1_{X_0^\times/k^\times} \ar[r]\ar[d]_{1-F} & \bigoplus_{ij} W_r\cur{O}_{D_{ij}} \ar[r]\ar[d]_{1-F} & 0 \\
0 \ar[r] &  \bigoplus_i W_{r-1}\Omega^1_{D_i} \ar[r] & W_{r-1}\omega^1_{X_0^\times/k^\times} \ar[r] & \bigoplus_{ij} W_{r-1}\cur{O}_{D_{ij}} \ar[r] & 0 } \]
with exact rows. Exactness of the middle vertical sequence at $W_r\omega^1_{X_0^\times/k^\times}$ now follows from the classical result \cite[Corollaire I.3.27, Proposition I.3.28]{Ill79} and a simple diagram chase.
\end{proof}

Next, we will need to understand the kernel of $W_r\omega^1_{X_0^\times,\log}\rightarrow W_r\omega^1_{X_0^\times/k^\times,\log}$.

\begin{lemma} \label{relnonrel} For all $r\geq1$ the sequence
\[ 0 \rightarrow \frac{\Z}{p^r\Z} \wedge d\log t \rightarrow W_r\omega^1_{X_0^\times,\log}\rightarrow W_r\omega^1_{X_0^\times/k^\times,\log} \rightarrow 0\]
is exact.
\end{lemma}

\begin{proof} Note that by \cite[Lemma 7.4]{Mat17} it suffices to show that 
\[ d\log(M_0) \cap W_r\cur{O}_{X_0}\wedge d\log t =  \frac{\Z}{p^r\Z}\wedge d\log t\]
inside $W_r\omega^1_{X_0^\times}$, the inclusion $\supset$ is clear. For the other, suppose that we are given an element of the form $g \wedge d\log t\in W_r\omega^1_{X_0^\times}$ which is in the image of $d\log$. Then we know that $\tilde{g}\wedge d\log t=d\log n+c$ in $W_{r+1}\omega^1_{X_0^\times}$, for some $c\in \ker \left( W_{r+1}\omega^1_{X_0^\times} \rightarrow W_{r}\omega^1_{X_0^\times}\right)$ and $\tilde{g}\in W_{r+1}\cur{O}_{X_0}$ lifting $g$. Arguing as in Proposition \ref{omfe} above we can see that $(1-F)(d\log n)=0$, and again applying \cite[Lemma 7.4]{Mat17} we can deduce that in fact $g=F(g)$ in $W_r\cur{O}_{X_0}$. Hence $g\in \Z/p^r\Z$ as claimed.
\end{proof}

Finally, we will need to know that the logarithmic de\thinspace Rham--Witt complex computes the log crystalline cohomology of the semistable scheme $\cur{X}$. To do so, we need to construct a suitable comparison morphism
\[ \mathbf{R}u_{\mathcal{X}^\times/W*} \mathcal{O}_{\mathcal{X}^\times/W}^\cris \isomto W\omega^*_{\mathcal{X}^\times}, \]
where $u_{\mathcal{X}^\times/W}: (\mathcal{X}^\times/W)_\cris \rightarrow \mathcal{X}_\et$ is the natural projection from the log-crystalline site of $\mathcal{X}^\times/W$ to the \'etale site of $\mathcal{X}$. Unfortunately, we cannot directly appeal to the construction of \cite[\S6]{Mat17}, since $\mathcal{X}$ is not of finite type over $W$. However, we can easily get round this by exploiting the fact that the log scheme $\spec{R^\times}$ has an obvious log-$p$-smooth lift over $W$, namely the scheme $\spec{W\pow{t}}$ together with the log structure $L_W$ defined by the divisor $t=0$. We therefore take an embedding system 
\[ \xymatrix{ \mathcal{X}_\bu^\times \ar[d] \ar[r] & (\mathcal{Y}_\bu,N_\bu) \ar[d] \\ \mathcal{X}^\times \ar[r] & \left( \spec{W\pow{t}},L_W\right) } \]
for the finite type morphism of log schemes $\mathcal{X}^\times \rightarrow \left(\spec{W\pow{t}},L_W\right)$ in the sense of \cite[Definition 6.3]{Mat17}, and then simply consider $\mathcal{X}^\times, \mathcal{X}_\bu^\times$ and $(\mathcal{Y}_\bu,N_\bu)$ instead as (simplicial) log schemes over $\spec{W}$, the latter being endowed with the trivial log structure. We now proceed exactly as in \cite[\S6]{Mat17}, or \cite[\S II.1]{Ill79} to produce the required comparison morphism 
\[ \mathbf{R}u_{\mathcal{X}^\times/W*} \mathcal{O}_{\mathcal{X}^\times/W}^\cris \isomto W\omega^*_{\mathcal{X}^\times}. \]

\begin{proposition}\label{cohcomp} The induced map
\[ H^i_{\log\text{-}\mathrm{cris}}(\cur{X}^{\times}/W) \rightarrow H^i_\mathrm{cont}(\cur{X}_\et,W\omega^*_{\cur{X}^{\times}}) 
\]
on cohomology is an isomorphism, for all $i\geq 0$.
\end{proposition}

\begin{proof} It suffices to show that $H^i_{\log\text{-}\mathrm{cris}}(\cur{X}^{\times}/W_r) \isomto H^i(\cur{X}_\et,W_r\omega^*_{\cur{X}^{\times}})$ for all $r$, where $W_r=W_r(k)$. Arguing locally on $\cur X$ we may assume in fact that $\cur X$ is affine, and in particular admits a closed embedding $\cur X \hookrightarrow \cur{P}$ into some affine space over $W_r\pow{t}$. Thus if we equip $\cur P$ with the log structure coming from the (smooth) divisor defined by $t=0$, the closed immersion $\cur X \hookrightarrow \cur{P}$ can be promoted to an exact closed immersion of log schemes.

Now applying N\'eron--Popescu desingularisation \cite[Theorem 1.8]{Pop86} to $W_r[t] \rightarrow W_r\pow{t}$, we may write $\cur P =\lim_\alpha P^\alpha$ as a limit of smooth $W_r[t]$-schemes, such that:
\begin{itemize} \item there exist compatible closed subschemes $X^\alpha \subset P^\alpha$, each of whose inverse image in $\cur P$ is precisely $\cur X$, an each of which is smooth over $k$;
\item the divisors $D^\alpha := X^\alpha \cap \left\{ t=0 \right\}$, each of whose inverse image in $\cur X$ is precisely the special fibre $X_0$, have normal crossings.
\end{itemize}
Both the log de\thinspace Rham--Witt complex and \'etale cohomology commute with cofiltered limits of schemes, thus by using \cite[Theorem 7.2]{Mat17} it suffices to show that the same is true of log-crystalline cohomology, in other words that we have
\[ H^i_{\log\text{-}\mathrm{cris}}(\cur{X}^{\times}/W_r)= \mathrm{colim}_\alpha H^i_{\log\text{-}\mathrm{cris}}(X_\alpha^{\times}/W_r), \]
where $X_\alpha^{\times}$ denotes the scheme $X_\alpha$ endowed with the log structure given by $D_\alpha$. By \cite[Theorem 6.4]{Kat89}, $H^i_{\log\text{-}\mathrm{cris}}(X_\alpha^{\times}/W_r)$ is computed as the de\thinspace Rham cohomology of the log-PD envelope of $X_\alpha^\times$ inside $P_\alpha$. Since log-PD envelopes commute with cofiltered limits of schemes (i.e. filtered colimits of rings), it suffices to show that $H^i_{\log\text{-}\mathrm{cris}}(\cur{X}^{\times}/W_r)$ can be computed as the de\thinspace Rham cohomology of the log-PD envelope of $\cur{X}^\times$ inside $\cur P$.

In other words, what we require a logarithmic analogue of \cite[Theorem 1.7]{Kat91}, or equivalently a log-$p$-basis analogue of \cite[Theorem 6.4]{Kat89}. But this follows from Proposition 1.6.6 of \cite{CV15}.\end{proof}

\section{Morrow's variational Tate conjecture for divisors}

The goal of this section is to offer a simpler proof of a special case of \cite[Theorem 3.5]{Mor14} for smooth and proper schemes $\cur X$ over the power series ring $R=k\pow{t}$. This result essentially states that a line bundle on the special fibre of $\cur X$ lifts iff its its first Chern class in $H^2_\mathrm{cris}$ does, and should be viewed as an equicharacteristic analogue of Berthelot and Ogus's theorem \cite[Theorem 3.8]{BO83} stating that a line bundle on the special fibre of a smooth proper scheme over a DVR in mixed characteristic lifts iff its Chern class lies in the first piece of the Hodge filtration. We will also give a slightly different interpretation of this result that emphasises the philosophy that in equicharacteristic the `correct' analogue of a Hodge filtration is an $\cur{E}^\dagger$-structure. Our proof is simpler in that it does not depend on any results from topological cyclic homology, but only on fairly standard properties of the de\thinspace Rham--Witt complex. As such, it is more readily adaptable to the semistable case, which we shall do in \S\ref{ssv} below.

Throughout this section, $\cur X$ will be a smooth and proper $R=k\pow{t}$-scheme. Let $R_n$ denote $k\pow{t}/(t^{n+1})$ and set $X_n=\cur{X} \otimes_R R_n$. Write $X$ for the generic fibre of $\cur X$ and $\mathfrak{X}$ for its formal ($t$-adic) completion. Since all schemes in this section will have trivial log structure, we will use the notation $W_\bu\Omega^*$ for the de\thinspace Rham--Witt complex instead of $W_\bu\omega^*$. The key technical calculation we will make is the following. 

\begin{lemma} \label{d1} Fix $n\geq 0$, write $n=p^mn_0$ with $(n_0,p)=1$, and let $r=m+1$. Then the map
\[ d\log :1+t^n\cur{O}_{X_n}\rightarrow W_r\Omega^1_{X_n,\log} \]
is injective.
\end{lemma}

\begin{proof} It suffices to prove the corresponding statement for sections on some open affine $\spec{A_n}\subset X_n$, which we may moreover assume to be \'etale over $R_n[x_1,\ldots,x_d]$. In this case, since deformations of smooth affine schemes are trivial, we have $A_n\cong A_0\otimes_k R_n$. Hence $1+t^nA_n=1+t^nA_0$, and our problem therefore reduces to showing that if $a\in A_0$ is such that $d\log[1+at^n]=0$, then in fact $a=0$. But vanishing of $a$ may be checked over all closed points of $\spec{A_0}$, so by functoriality of the $d\log$ map we may in fact assume that $A_0$ is a finite extension of $k$; enlarging $k$ we may moreover assume that $A_0=k$. In other words we need to show that the map
\[ d\log: 1+t^nk  \rightarrow W_r\Omega^1_{R_n}\]
is injective. Since $k$ is perfect, any $1+at^n\in 1+t^nk$ can be written uniquely as $(1+t^{n_0}b)^{p^m}$ for some $b\in k$, hence $d\log[1+at^n]=p^md\log(1+t^{n_0}b)$. It follows that if $d\log[1+at^n]=0$, then $n_0p^mbt^{n_0-1}dt=0$ in $W_r\Omega^1_{R_n}$  - note that although $b\in k$, nonetheless $p^mb$ still makes sense as an element of $W_{m+1}(k)=W_r(k)$. Since any non-zero such $b$ is invertible, the lemma will follow if we can show that $p^mt^{n_0-1}dt$ is non-zero in $W_r\Omega^1_{R_n}$. This can be checked easily using the exact sequence
\[ \frac{W_r((t^{n+1}))}{W_r((t^{n+1})^2)} \overset{d}{\rightarrow} W_r\Omega^1_{k[t]} \otimes_{W_r(k[t])} W_rR_n\rightarrow W_r\Omega^1_{R_n}\rightarrow 0\]
from \cite{LZ05}.
\end{proof}

From this we deduce the following.

\begin{proposition} \label{e1} For $r\gg0$ (depending on $n$) there is a commutative diagram
\[ \xymatrix{ 1 \ar[r] & 1+t\cur{O}_{X_n} \ar[r] \ar@{=}[d] & \cur{O}_{X_n}^*  \ar[r] \ar[d]^{d\log}& \cur{O}_{X_0}^* \ar[d]^{d\log}\ar[r] & 1 \\ 
1 \ar[r] & 1+t\cur{O}_{X_n} \ar[r]^-{d\log} &  W_r\Omega^1_{X_n,\log} \ar[r] & W_r\Omega^1_{X_0,\log} \ar[r] & 0  }\]
with exact rows.
\end{proposition}

\begin{proof} It is well-known that the top row is exact, and the diagram is clearly commutative, it therefore suffices to show that for all $n$ the sequence 
\[ 1 \rightarrow  1+t\cur{O}_{X_n} \rightarrow  W_r\Omega^1_{X_n,\log} \rightarrow W_r\Omega^1_{X_0,\log} \rightarrow 0 \]
is exact for $r\gg 0$. From the definition of $W_r\Omega^1_{X_n,\log}$ and the exactness of the sequence
\[1 \rightarrow  1+t\cur{O}_{X_n}\rightarrow  \cur{O}_{X_n}^*  \rightarrow  \cur{O}_{X_0}^* \rightarrow  1  \]
it is immediate that $W_r\Omega^1_{X_n,\log} \rightarrow W_r\Omega^1_{X_0,\log} $ is surjective and the composite $1+t\cur{O}_{X_n} \rightarrow W_r\Omega^1_{X_0,\log} $ is zero. Given $\alpha\in \cur{O}_{X_n}^*$ mapping to $0$ in $W_r\Omega^1_{X_0,\log}$, it follows from \cite[Proposition I.3.23.2]{Ill79} that there exists $\beta\in \cur{O}_{X_n}^*$ and $\gamma\in 1+t\cur{O}_{X_n}$ such that $\alpha=\beta^{p^r}+\gamma$, and hence $d\log \alpha=d\log \gamma$ in $W_r\Omega^1_{X_n,\log}$. The sequence
\[ 1+t\cur{O}_{X_n} \rightarrow  W_r\Omega^1_{X_n,\log} \rightarrow W_r\Omega^1_{X_0,\log} \rightarrow 0 \]
is therefore exact, and it remains to show that
\[ d\log: 1+t\cur{O}_{X_n} \rightarrow  W_r\Omega^1_{X_n,\log} \]
is injective for $r\gg0$. By induction on $n$ this follows from Lemma \ref{d1} above. 
\end{proof}
 
We now set
\[ W_r\Omega^i_{\mathfrak{X},\log}:=\lim_n W_r\Omega^i_{X_n,\log}\]
as sheaves on $\mathfrak{X}_{\et}$ and define
\[H^j_{\mathrm{cont}}(\mathfrak{X}_{\et},W\Omega^i_{\mathfrak{X},\log}):=H^j(\mathbf{R}\lim_r \mathbf{R}\Gamma(\mathfrak{X}_{\et},W_r\Omega^i_{\mathfrak{X},\log})) .\] 
As an essentially immediate corollary of Proposition \ref{e1}, we deduce the key step of Morrow's proof of the variational Tate conjecture in this case.

\begin{corollary} \label{m1} Let $\cur{L}\in \mathrm{Pic}(X_0)$, with first Chern class $c_1(\cur L)\in H^1_\mathrm{cont}(X_{0,\et},W\Omega^1_{X_0,\log})$. Then $\cur L$ lifts to $\mathrm{Pic}(\cur X)$ if and only if $c_1(\cur L)$ lifts to $H^1_{\mathrm{cont}}(\cur{X}_{\et},W\Omega^1_{\cur{X},\log})$.
\end{corollary}

\begin{proof} One direction is obvious. For the other direction, assume that the first Chern class $c_1(\cur L)$ lifts to $H^1_{\mathrm{cont}}(\cur{X}_{\et},W\Omega^1_{\cur{X},\log})$, in particular it therefore lifts to $H^1_{\mathrm{cont}}(\mathfrak{X}_{\et},W\Omega^1_{\mathfrak{X},\log})$. Hence by Proposition \ref{e1} it follows that $\cur L$ lifts to $\mathrm{Pic}(\mathfrak{X})$, and we may conclude using Grothendieck's algebrisation theorem that it lifts to $\mathrm{Pic}(\cur X)$.
\end{proof}

From this the \textbf{(crys-}$\mathbf{\phi}\textbf{)}$ form of the variational Tate conjecture follows as in \cite{Mor14}.

\begin{corollary} \label{mm1} Let $\cur{L}\in \mathrm{Pic}(X_0)_{\Q}$, with first Chern class $c_1(\cur L)\in H^2_\mathrm{cris}(X_0/K)^{\varphi=p}$. Then $\cur L$ lifts to $\mathrm{Pic}(\cur X)_{\Q}
$ if and only if $c_1(\cur L)$ lifts to $H^2_\mathrm{cris}(\cur{X}/K)^{\varphi=p}$. 
\end{corollary}

\begin{proof} Let us first assume that $k$ is algebraically closed. By \cite[Proposition 3.2]{Mor14} the inclusions $W\Omega^1_{\cur{X},\log}[-1]\rightarrow W\Omega^*_{\cur{X},\log}$ and $W\Omega^1_{X_0,\log}[-1]\rightarrow W\Omega^*_{X_0,\log}$ induce an isomorphism 
\[ H^1_\mathrm{cont}(X_{0,\et},W\Omega^1_{X_0,\log})_{\Q}\isomto H^2_\mathrm{cris}(X_0/K)^{\varphi=p} \]
and a surjection
\[ H^1_{\mathrm{cont}}(\cur{X}_{\et},W\Omega^1_{\cur{X},\log} )_{\Q} \twoheadrightarrow H^2_\mathrm{cris}(\cur{X}/K)^{\varphi=p}. \]
The claim follows. In general, we argue as in \cite[Theorem 1.4]{Mor14}: the claim for $k$ algebraically closed shows that $\cur{L}$ lifts to $\mathrm{Pic}(\cur{X})_{\Q}$ after making the base change $k\pow{t}\rightarrow \overline{k}\pow{t}$. Let $k\pow{t}^\mathrm{sh}$ denote the strict Henselisation of $k\pow{t}$ inside $\overline{k}\pow{t}$, by N\'eron--Popescu desingularisation there exists some smooth local $k\pow{t}^\mathrm{sh}$-algebra $A$ such that $\cur{L}$ lifts to $\mathrm{Pic}(\cur{X})_{\Q}$ after making the base change $k\pow{t}\rightarrow A$. But the map $k\pow{t}^\mathrm{sh}\rightarrow A$ has a section, from which it follows that in fact $\cur{L}$ lifts to $\mathrm{Pic}(\cur{X})_{\Q}$ after making some finite field extension $k\rightarrow k'$. But now simply taking the pushforward via $\cur{X}\otimes_k k' \rightarrow \cur{X}$ and dividing by $[k':k]$ gives the result.
\end{proof}

To finish off this section, we wish to give a slightly different formulation of Corollary \ref{mm1}. After \cite{LP16} we can consider the `overconvergent' rigid cohomology $H^i_\rig(X/\cur{E}^\dagger)$ of the generic fibre $X$, which is a $(\varphi,\nabla)$-module over the bounded Robba ring $\cur{E}^\dagger$. Set $H^i_\rig(X/\cur{R}):=H^i_\rig(X/\cur{E}^\dagger)\otimes_{\cur{E}^\dagger} \cur{R}$. By combining Dwork's trick with smooth and proper base change in crystalline cohomology we have an isomorphism
\[  H^i_\rig(X/\cur{R})^{\nabla=0}\cong H^i_\rig(X_0/K) \]
for all $i$. In particular, for any $\cur L\in \mathrm{Pic}(X_0)_{\Q}$ we can consider $c_1(\cur{L})$ as an element of $H^i_\rig(X/\cur{R})^{\nabla=0}\subset H^i_\rig(X/\cur{R})$. One of the general philosophies of $p$-adic cohomology in equicharacteristic is that while the cohomology groups $H^i_\rig(X/\cur{R})$ in some sense only depend on the special fibre $X_0$, the `lift' $X$ of $X_0$ is seen in the $\cur{E}^\dagger$-lattice $H^i_\rig(X/\cur{E}^\dagger)\subset  H^i_\rig(X/\cur{R})$. The correct equicharacteristic analogue of a Hodge filtration, therefore, is an $\cur{E}^\dagger$-structure. With this in mind, then, a statement of the variational Tate conjecture for divisors which is perhaps slightly more transparently analogous to that in mixed characteristic is the following.

\begin{theorem}\label{l11} Assume that $\cur{X}$ is projective over $R$. Then a line bundle $\cur L\in \mathrm{Pic}(X_0)_{\Q}$ lifts to $\mathrm{Pic}(\cur X)_{\Q}$ if and only if  $c_1(\cur L)\in H^2_\rig(X/\cur{R})$ lies in $H^2_\rig(X/\cur{E}^\dagger)$.
\end{theorem}

\begin{proof} This is simply another way of stating the condition $\textbf{(flat)}$ in \cite[Theorem 3.5]{Mor14}.
\end{proof}

\begin{remark} It seems entirely plausible that the methods of this section can be easily adapted to give a proof of \cite[Theorem 3.5]{Mor14} in general, i.e. over $k\pow{t_1,\ldots,t_n}$ rather than just $k\pow{t}$.
\end{remark}

\section{A semistable variational Tate conjecture for divisors}\label{ssv}

In this section we will prove a semistable version of Theorem \ref{l11}, or equivalently an equicharacteristic analogue of \cite[Theorem 0.1]{Yam11}. The basic set-up will be to take a proper, semistable scheme $\cur{X}/R$, as before we will consider the semistable schemes $X_n/R_n$ as well as the smooth generic fibre $X/F$. We will also let $\mathfrak{X}$ denote the formal completion of $\mathcal{X}$. 

The special fibre of $\cur X$ defines a log structure $M$, and pulling back via the immersion $X_n\rightarrow \cur X$ defines a log structure $M_{n}$ on each $X_n$. For each $n$ we will put a log structure $L_n$ on $R_n$ via $\N\rightarrow R_n$, $1\mapsto t$, note that for $n=0$ this is the log structure of the punctured point on $k$. We will let $L$ denote the log structure on $R$ defined by the same formula. As before we will write $R^\times=(R,L)$, $R_n^\times=(R_n,L_n)$, $\cur{X}^\times=(\cur{X},M)$, $X_n^\times=(X_n,M_n)$ and $k^\times=(k,L_0)$. The logarithmic version of Proposition \ref{e1} is then the following.

\begin{proposition} For $r\gg0$ (depending on $n$) there is a commutative diagram
\[ \xymatrix{ 1 \ar[r] & 1+t\cur{O}_{X_n} \ar[r] \ar@{=}[d] & \cur{O}_{X_n}^*  \ar[r] \ar[d]& \cur{O}_{X_0}^* \ar[d]\ar[r] & 1 \\ 
1 \ar[r] & 1+t\cur{O}_{X_n} \ar[r] \ar[d] & M_n^\mathrm{gp}  \ar[r] \ar[d]^{d\log}& M_0^\mathrm{gp} \ar[d]^{d\log}\ar[r] & 0 \\
1 \ar[r] & \cur{K}_{n,r} \ar[r] &  W_r\omega^1_{X_n^\times,\log} \ar[r] & W_r\omega^1_{X_0^\times/k^\times,\log} \ar[r] & 0  }\]
with exact rows. Moreover each $\cur{K}_{n,r}$ fits into an exact sequence of pro-sheaves on $X_{n,\et}$
\[ 1\rightarrow 1+t\cur{O}_{X_n} \rightarrow \{ \cur{K}_{n,r} \}_r \rightarrow  \{ \Z/p^r\Z \}_r \rightarrow 0\]
which is split compatibly with varying $n$.
\end{proposition}

\begin{proof} We first claim that if we replace $W_r\omega^1_{X_0^\times/k^\times,\log}$ by $W_r\omega^1_{X_0^\times,\log}$ then we obtain an exact sequence
\[ 1\rightarrow 1+t\cur{O}_{X_n} \rightarrow  W_r\omega^1_{X_n^\times,\log} \rightarrow W_r\omega^1_{X_0^\times,\log} \rightarrow 0 \]
for $r\gg0$. Using Proposition \ref{fesl1} the proof of the exactness of 
\[ 1+t\cur{O}_{X_n} \rightarrow  W_r\omega^1_{X_n^\times,\log} \rightarrow W_r\omega^1_{X_0^\times,\log} \rightarrow 0 \] 
is exactly as in Proposition \ref{e1}. In fact, to check exactness on the left we can even apply Proposition \ref{e1}: to check a section of $1+t\cur{O}_{X_n}$  vanishes it suffices to do on a dense open subscheme of $X_n$, we may therefore \'etale locally replace $X_n$ by the canonical thickening of the smooth locus of the special fibre. But now we are in the smooth case, so we apply Proposition \ref{e1} (which holds locally).

Applying Lemma \ref{relnonrel} we know that the kernel of 
\[W_r\omega^1_{X_0^\times,\log} \rightarrow W_r\omega^1_{X_0^\times/k^\times,\log} \]
is isomorphic to $\Z/p^r\Z$, generated by $d\log t$. The snake lemma then shows that, defining $\cur{K}_{n,r}$ to be the kernel of $W_r\omega^1_{X_n^\times,\log}\rightarrow W_r\omega^1_{X_0^\times/k^\times,\log}$, we have the exact sequence
\[ 1\rightarrow 1+t\cur{O}_{X_n} \rightarrow \cur{K}_{n,r} \rightarrow  \Z/p^r\Z  \rightarrow 0 \]
for $r\gg0$. To see that it splits compatibly with varying $r$ and $n$ it therefore suffices to show that there exist compatible classes $\omega_r\in W_r\omega^1_{X_n^\times}$ whose image in $W_r\omega^1_{X_0^\times,\log}$ generate the kernel of $W_r\omega^1_{X_0^\times,\log} \rightarrow W_r\omega^1_{X_0^\times/k^\times,\log}$; as we have already observed the classes of $d\log t$ will suffice.
\end{proof}

Let $\mathrm{Pic}(X_0^\times)=H^1(X_{0,\et},M_0^\mathrm{gp})$ and $\mathrm{Pic}(\cur X^\times)=H^1({\cur X}_{\et},M^\mathrm{gp})$. As before, we therefore obtain the following.

\begin{corollary} Let $\cur L\in \mathrm{Pic}(X_0^\times)$ (resp. $\mathrm{Pic}(X_0)$). Then $\cur{L}$ lifts to $\mathrm{Pic}(\cur{X}^\times)$ (resp. $\mathrm{Pic}(\cur X)$) iff $c_1(\cur L)\in H^1_\mathrm{cont}(X_{0,\et},W\omega^1_{X_0^\times/k^\times,\log})$ lifts to $H^1_\mathrm{cont}({\cur X}_{\et},W\omega^1_{\cur X^\times,\log})$.
\end{corollary}

\begin{proof} This is similar to the proof of Corollary \ref{m1}, although a little more care is needed in taking the limits in $n$ and $r$. Again, one direction is clear, so we assume that we are given a (logarithmic) line bundle whose Chern class lifts. First we note that we have isomorphisms
\[ \mathrm{Pic}(\mathfrak{X}) \cong H^1_\mathrm{cont}(X_{0,\et},\{\cur{O}_{X_n}^*\}_n),\;\; \mathrm{Pic}(\mathfrak{X}^\times) \cong H^1_\mathrm{cont}(X_{0,\et},\{M_n^\mathrm{gp} \}_n)\]
and hence the obstruction to lifting (in either case) can be viewed as an element of $H^2_\mathrm{cont}(X_{0,\et},\{1+t\cur{O}_{X_n}\}_n)$. The fact that the Chern class lifts implies that this obstruction vanishes in
\[ H^2_\mathrm{cont}(X_{0,\et}, \{ \cur{K}_{n,r}\}_{n,r}):=H^2(\mathbf{R}\lim_n\mathbf{R}\lim_r \mathbf{R}\Gamma(X_{0,\et},\cur{K}_{n,r}))\]
and hence the fact that the exact sequence of pro-sheaves
\[ 1\rightarrow 1+t\cur{O}_{X_n} \rightarrow \{ \cur{K}_{n,r} \}_r \rightarrow  \{ \Z/p^r\Z \}_r \rightarrow 0 \]
splits, compatibly with varying $n$, shows that the obstruction must itself vanish in $H^2_\mathrm{cont}(X_{0,\et},\{1+t\cur{O}_{X_n}\}_n)$. Finally, we need to see that we have isomorphisms $\mathrm{Pic}(\mathfrak{X})\cong \mathrm{Pic}(\cur X)$ and $\mathrm{Pic}(\mathfrak{X}^\times)\cong \mathrm{Pic}(\cur{X}^\times)$. The first is Grothendieck's algebrization theorem, to see the second we note that $\mathrm{Pic}(\cur{X}^\times) \cong \mathrm{Pic}(X)$, the Picard group of the generic fibre of $\cur X$, similarly $\mathrm{Pic}(\mathfrak{X}^\times)\cong \mathrm{Pic}(X^{\mathrm{an}})$, the Picard group of its analytification. The two are isomorphic by rigid analytic GAGA.
\end{proof}

To relate this to log crystalline cohomology, we use the following.

\begin{lemma}The inclusions $W_r\omega^1_{\cur{X}^\times,\log}[-1] \rightarrow W_r\omega^*_{\cur{X}^\times}$ and $W_r\omega^1_{X_0^\times/k^\times,\log}[-1] \rightarrow W_r\omega^*_{X_0^\times/k^\times}$ induce surjections
\begin{align*} H^1_\mathrm{cont}({\cur X}_{\et},W\omega^1_{\cur{X}^\times,\log})_{\Q} &\twoheadrightarrow H^2_{\log\text{-}\mathrm{cris}}(\cur{X}^\times/K)^{\varphi=p} \\
H^1_\mathrm{cont}({X_0}_{\et},W\omega^1_{X_0^\times/k^\times,\log})_{\Q} &\twoheadrightarrow H^2_{\log\text{-}\mathrm{cris}}(X_0^\times/K^\times)^{\varphi=p}
\end{align*}
where $\varphi$ is the semilinear Frobenius operator. If $k$ is algebraically closed, then the latter is in fact an isomorphism.
\end{lemma}

\begin{proof} Let us first consider $\cur{X}^\times$. Define the map $\cur{F}:\{ W_r\omega^{\geq 1}_{\cur{X}^\times}\}_r\rightarrow \{ W_r\omega^{\geq 1}_{\cur{X}^\times}\}_r$ to be $p^{i-1}F$ in degree $i$, note that in degrees $>1$ it is a contracting operator, and hence $1-\cur{F}$ is invertible on $W_r\omega^{> 1}_{\cur{X}^\times}$. Similarly, the map $1-V:\{W_r\cur{O}_\cur{X} \}_r\rightarrow \{W_r\cur{O}_\cur{X} \}_r$ is an isomorphism. From this and Proposition \ref{omfe} it follows that the triangle
\[ 0 \rightarrow \{W_r\omega^{1}_{\cur{X}^\times,\log} \}_r \rightarrow \{ W_r\omega^{\geq 1}_{\cur{X}^\times}\}_r \overset{1-\cur F}{\rightarrow}\{ W_r\omega^{\geq 1}_{\cur{X}^\times}\}_r\rightarrow 0 \]
of complexes of pro-sheaves is exact. Since $p\cur{F}=\varphi$ on  $W_r\omega^{\geq 1}_{\cur{X}^\times}$, we deduce an exact sequence
\[ 0\rightarrow \frac{H^1_\mathrm{cont}(\cur{X}_\et,W\omega^{\geq 1}_{\cur{X}^\times})_{\Q}}{\mathrm{im}(\varphi-p)} \rightarrow H^1_\mathrm{cont}({\cur X}_{\et},W\omega^1_{\cur{X}^\times,\log})_{\Q} \rightarrow H^2_\mathrm{cont}(\cur{X}_\et,W\omega^{\geq 1}_{\cur{X}^\times})^{\varphi=p}_{\Q} \rightarrow 0.\]
For a complex of $K$-modules $C^*$ with semilinear Frobenius, let us write $\mathbf{R}_{\varphi=p}(C^*)$ for the mapping cone $\mathrm{Cone}(C^* \overset{\varphi-p}{\rightarrow}C^*)$, and $H^n_{\varphi=p}(C^*)$ for its cohomology groups. Then since $1-V=$``$1-p\varphi^{-1}$'' is invertible on $\{W_r\cur{O}_\cur{X} \}_r$ we deduce that 
\[ \mathbf{R}_{\varphi=p}(\mathbf{R}\Gamma_\mathrm{cont}(\cur{X}_\et,W\omega^{\geq 1}_{\cur{X}^\times})_{\Q}) \cong \mathbf{R}_{\varphi=p}(\mathbf{R}\Gamma_\mathrm{cont}(\cur{X}_\et,W\omega^*_{\cur{X}^\times})_{\Q}). \]
From this we extract the diagram
\[ \xymatrix{  0\ar[r] & \frac{H^1_\mathrm{cont}(\cur{X}_\et,W\omega^{\geq 1}_{\cur{X}^\times})_{\Q}}{\mathrm{im}(\varphi-p)} \ar[r]\ar[d] & H^2_{\varphi=p}(\mathbf{R}\Gamma_\mathrm{cont}(\cur{X}_\et,W\omega^{\geq 1}_{\cur{X}^\times})_{\Q}) \ar[r]\ar[d] & H^2_\mathrm{cont}(\cur{X}_\et,W\omega^{\geq 1}_{\cur{X}^\times})^{\varphi=p}_{\Q} \ar[r]\ar[d] & 0 \\
0\ar[r] & \frac{H^1_\mathrm{cont}(\cur{X}_\et,W\omega^*_{\cur{X}^\times})_{\Q}}{\mathrm{im}(\varphi-p)} \ar[r] & H^2_{\varphi=p}(\mathbf{R}\Gamma_\mathrm{cont}(\cur{X}_\et,W\omega^*_{\cur{X}^\times})_{\Q}) \ar[r] & H^2_\mathrm{cont}(\cur{X}_\et,W\omega^*_{\cur{X}^\times})^{\varphi=p}_{\Q} \ar[r] & 0 }\]
with exact rows, such that the middle vertical arrow is an isomorphism. In particular, the right vertical arrow is an surjection, and applying Proposition \ref{cohcomp} we see that the map
\[  H^1_\mathrm{cont}({\cur X}_{\et},W\omega^1_{\cur{X}^\times,\log})_{\Q} \twoheadrightarrow H^2_{\log\text{-}\mathrm{cris}}(\cur{X}^\times/K)^{\varphi=p} \]
is surjective as claimed. An entirely similar argument works for $X_0^\times$, replacing Proposition \ref{cohcomp} with \cite[Theorem 7.9]{Mat17}, and in fact shows that 
\[ H^1_\mathrm{cont}({X_0}_{\et},W\omega^1_{X_0^\times/k^\times,\log})_{\Q} \twoheadrightarrow H^2_{\log\text{-}\mathrm{cris}}(X_0^\times/K^\times)^{\varphi=p} \]
is an isomorphism if and only if $(\varphi-p)$ is surjective on $H^1_{\log\text{-}\mathrm{cris}}(X_0^\times/K^\times)$. If $k$ is algebraically closed, this follows from semisimplicity of the category of $\varphi$-modules over $K$.
\end{proof}

This enables us to deduce the following.

\begin{corollary} \label{mm2} Let $\cur L\in \mathrm{Pic}(X_0^\times)_{\Q}$ (resp. $\mathrm{Pic}(X_0)_{\Q}$). Then $\cur{L}$ lifts to $\mathrm{Pic}(\cur{X}^\times)_{\Q}$ (resp. $\mathrm{Pic}(\cur X)_{\Q}$) iff $c_1(\cur L)\in H^2_{\log\text{-}\mathrm{cris}}(X_0^\times/K^\times)^{\varphi=p}$ lifts to $H^2_{\log\text{-}\mathrm{cris}}(\cur X^\times/K)^{\varphi=p}$.
\end{corollary}

\begin{proof} Exactly as in the proof of Corollary \ref{mm1}.
\end{proof}

Let us now rephrase this more closely analogous to Yamashita's criterion in \cite{Yam11}. Note that thanks to \cite[Corollary 5.8]{LP16} we have an isomorphism
\[ H^i_\rig(X/\cur{R})\cong H^i_{\mathrm{log}\text{-}\mathrm{cris}}(X_0^\times/K^\times)\otimes \cur{R}\]
of $(\varphi,\nabla)$-modules over $\cur{R}$, which induces an isomorphism
\[ H^i_\rig(X/\cur{R})^{\nabla=0} \cong H^i_{\mathrm{log}\text{-}\mathrm{cris}}(X_0^\times/K^\times)^{N=0}. \]
By \cite[Proposition 2.2]{Yam11} (whose proof does not use the existence of a lift to characteristic $0$), the first Chern class $c_1(\cur L)$ of any $\cur L$ in $\mathrm{Pic}(X_0^\times)_{\Q}$ or $\mathrm{Pic}(X_0)_{\Q}$ satisfies $N(c_1(\cur L))=0$. Hence we may view $c_1(\cur{L})$ as an element of $H^2_\rig(X/\cur{R})$.

\begin{theorem} \label{mainss} Assume that $\cur{X}$ is projective over $R$. Then $\cur{L}$ lifts to $\mathrm{Pic}(\cur{X}^\times)_{\Q}$ (resp. $\mathrm{Pic}(\cur X)_{\Q}$) iff $c_1(\cur{L})\in H_\rig^2(X/\cur{E}^\dagger)\subset H_\rig^2(X/\cur{R})$.
\end{theorem}

\begin{proof} Note that if $c_1(\cur{L})\in H_\rig^2(X/\cur{E}^\dagger)$, it is automatically in the subspace $H_\rig^2(X/\cur{E}^\dagger)^{\nabla=0,\varphi=p}$. Now consider the Leray spectral sequence for log crystalline cohomology
\[ E_2^{p,q}= H^q_{\log\text{-}\mathrm{cris}}\left(\spec{R^\times},\mathbf{R}^pf_*\cur{O}^\mathrm{cris}_{\cur X^\times/K}\right) \Rightarrow H^{p+q}_{\log\text{-}\mathrm{cris}}(\cur{X}^\times/K),\]
where $f:\cur{X}^\times\rightarrow \spec{R^\times}$ denotes the structure map. Since $\cur X$ is projective we obtain maps
\[ u^i:\mathbf{R}^{d-i}f_*\cur{O}^\mathrm{cris}_{\cur{X}^\times/K} \rightarrow \mathbf{R}^{d+i}f_*\cur{O}^\mathrm{cris}_{\cur{X}^\times/K} \]
of log-$F$-isocrystals over $R^\times$ by cupping with the class of a hyperplane section, we claim that $u^i$ is an isomorphism. To check this, we note that we can identify the category of log-$F$-isocrystals over $R^\times$ with the category $\underline{\mathbf{M}\Phi}_{\cur{E}^+}^{\nabla,\log}$ of log-$\pn$-modules over the ring $\cur{E}^+:=W\pow{t}\otimes_W K$ as considered in \cite[\S5.3]{LP16}. We now note that the functor of `passing to the generic fibre', i.e. tensoring with $\cur{E}:=\cur{E}^+\tate{t^{-1}}$ is fully faithful, by \cite[Theorem 5.1]{Ked04c} (together with a simple application of the 5 lemma), and hence by the hard Lefschetz theorem in rigid cohomology \cite{Car16} (together with standard comparison theorems in crystalline cohomology) the isomorphy of $u^i$ follows. Hence applying the formalism of \cite[\S2]{Mor14} we obtain surjective maps
\begin{align*}	H^2_{\log\text{-}\mathrm{cris}}(\cur{X}^\times/K) &\rightarrow H^0_{\mathrm{log}\text{-}\mathrm{cris}}\left(\spec{R^\times},\mathbf{R}^2f_*\cur{O}^\mathrm{cris}_{\cur{X}^\times/K}\right) \\
H^2_{\log\text{-}\mathrm{cris}}(\cur{X}^\times/K)^{\varphi=p} &\rightarrow H^0_{\mathrm{log}\text{-}\mathrm{cris}}\left(\spec{R^\times},\mathbf{R}^2f_*\cur{O}^\mathrm{cris}_{\cur{X}^\times/K}\right)^{\varphi=p}
\end{align*}
as the edge maps of degenerate Leray spectral sequences (see in particular \cite[Lemma 2.4, Theorem 2.5]{Mor14}). Finally we note that again applying Kedlaya's full faithfulness theorem, together with the proof of \cite[Proposition 5.45]{LP16}, we can see that 
\[ H^0_{\mathrm{log}\text{-}\mathrm{cris}}\left(\spec{R^\times},\mathbf{R}^2f_*\cur{O}^\mathrm{cris}_{\cur{X}^\times/K}\right)^{\varphi=p}\cong H_\rig^2(X/\cur{E}^\dagger)^{\nabla=0,\varphi=p}\]
and the claim follows.
\end{proof}

We will now give one final reformulation of this result.

\begin{definition} \begin{enumerate} \item We say that a cohomology class in $H^2_\rig(X/\cur{E}^\dagger)$ is \emph{algebraic} if it is in the image of $\mathrm{Pic}(X)_{\Q}$ under the Chern class map.
\item We say that a cohomology class in $H^2_{\log\text{-}\mathrm{cris}}(X_0^\times/K)$ is \emph{log-algebraic} if it is in the image of $\mathrm{Pic}(X_0^\times)_{\Q}$ under the Chern class map.
\item We say that a cohomology class in $H^2_{\log\text{-}\mathrm{cris}}(X_0^\times/K)$ is \emph{algebraic} if it is in the image of $\mathrm{Pic}(X_0)_{\Q}$ under the Chern class map.
\end{enumerate}
\end{definition}

Let
\[ \mathrm{sp}:H^2_\rig(X/\cur{E}^\dagger)^{\nabla=0}\hookrightarrow H^2_\rig(X/\cur{R})^{\nabla=0} \isomto H^2_{\log\text{-}\mathrm{cris}}(X_0^\times/K)^{N=0}\hookrightarrow  H^2_{\log\text{-}\mathrm{cris}}(X_0^\times/K)\]
denote the composite homomorphism.

\begin{theorem} \label{algeq} Assume that $\cur X$ is projective, and let $\alpha\in H^2_\rig(X/\cur{E}^\dagger)$. The following are equivalent.
\begin{enumerate}
\item $\alpha$ is algebraic.
\item $\nabla(\alpha)=0$ and $\mathrm{sp}(\alpha)$ is log-algebraic.
\item $\nabla(\alpha)=0$ and $\mathrm{sp}(\alpha)$ is algebraic.
\end{enumerate}
\end{theorem}

\begin{proof} Note that since $\mathrm{sp}$ is injective, the hypotheses in both (2) and (3) imply that $\varphi(\alpha)=p\alpha$. Since $\mathcal{X}$ is flat, its special fibre is a principal Cartier divisor, therefore the restriction map $\mathrm{Pic}(\cur X)_{\Q}\rightarrow \mathrm{Pic}(X)_{\Q}$ is as isomorphism. The claim then follows from Theorem \ref{mainss}.
\end{proof}

\section{Global results}\label{glob}

In this section we will deduce some global algebraicity results more closely analogous to the main results of \cite{Mor14}. We will therefore change notation and let $F$ denote a function field of transcendence degree one over our perfect field $k$ of characteristic $p$. We will let $v$ denote a place of $F$ with completion $F_v$ and residue field $k_v$. Let $\cur{C}$ denote the unique smooth, proper, geometrically connected curve over $k$ with function field $F$. Let $F^\mathrm{sep}$ denote a fixed separable closure of $F$ with Galois group $G_F$.

\begin{definition} Define $F\text{-}\mathrm{Isoc}(F/K):=2\text{-}\mathrm{colim}_U F\text{-}\mathrm{Isoc}(U/K)$, the colimit being taken over all non-empty open subschemes $U\subset \cur{C}$.  
\end{definition}

Note that by \cite[Theorem 5.2.1]{Ked07}, for any $E\in F\text{-}\mathrm{Isoc}(F/K)$, defined on some $U\subset \cur{C}$, the zeroeth cohomology group
\[ E^{\nabla=0}= H^0_\rig(U/K,E) \]
is a well-defined (i.e. independent of $U$) $F$-isocrystal over $K$. For any smooth and \emph{projective} variety $X/F$ we have cohomology groups $\cur{H}^i_\rig(X/K)\in F\text{-}\mathrm{Isoc}(F/K)$ obtained by choosing a smooth projective model over some $U\subset \cur{C}$, taking the higher direct images and applying \cite[Corollaire 3]{MT04}. As constructed in \cite[\S6]{Pal15b} (see in particular Propositions 6.17 and 7.2) there is a $p$-adic Chern class map
\[ c_1:\mathrm{Pic}(X)_{\Q} \rightarrow \cur{H}^2_\rig(X/K)^{\nabla=0} \]
and we will call elements in the image \emph{algebraic}.

Assume now that $X$ has semistable reduction at $v$, denote the associated log smooth scheme over $k_v^\times$ by $X_v^\times$. Let $\cur{E}^\dagger_v$ denote a copy of the bounded Robba ring `at $v$', so that by \cite[\S6.1]{Tsu98} there is a functor
\[ \mathbf{i}_v^*: F\text{-}\mathrm{Isoc}(F/K) \rightarrow \underline{\mathbf{M}\Phi}_{\cur{E}^\dagger_v}^\nabla.  \]
Thanks to the proof of \cite[Proposition 5.52]{LP16} this functor sends $\cur{H}^2_\rig(X/K)$ to $H^2_\rig(X_{F_v}/\cur{E}^\dagger_v)$. In particular we obtain a map
\[ r_v: \cur{H}^2_\rig(X/K)^{\nabla=0}\rightarrow H^2_\rig(X_{F_v}/\cur{E}^\dagger_v)^{\nabla=0}\]
and composing with the specialisation map considered at the end of \S\ref{ssv} we obtain a homomorphism
\[ \mathrm{sp}_v:\cur{H}^2_\rig(X/K)^{\nabla=0}\rightarrow H^2_{\log\text{-}\mathrm{cris}}(X_v^\times/K_v^\times) \]
where $K_v=W(k_v)[1/p]$.

\begin{theorem} \label{glma}Assume that $X$ is projective, and let $\alpha\in \cur{H}^2_\rig(X/K)^{\nabla=0}$. The following are equivalent.
\begin{enumerate} \item $\alpha$ is algebraic.
\item $\mathrm{sp}_v(\alpha)$ is algebraic.
\item $\mathrm{sp}_v(\alpha)$ is log-algebraic.
\end{enumerate}
\end{theorem}

\begin{proof} As before the hypotheses in (2) and (3) imply that $\varphi(\alpha)=p\alpha$. By Theorem \ref{algeq} we clearly have $(1)\Rightarrow (2)\Leftrightarrow (3)$, and if $(2)$ or $(3)$ hold then there exists a line bundle $\cur L\in \mathrm{Pic}(X_{F_v})_{\Q}$ such that $r_v(\alpha)=c_1(\cur L)$ in $H^2_\rig(X_{F_v}/\cur{E}^\dagger_v)^{\nabla=0}$. To descend $\cur L$ to $\mathrm{Pic}(X)_{\Q}$ we follow the proof of  Corollary \ref{mm1}. Specifically, applying N\'eron--Popescu desingularisation to the extension $F_v^h\rightarrow F_v$ from the Henselisation to the completion at $v$ and arguing exactly as before we can in fact assume that $\cur L$ descends to $X_{F_v^h}$, and hence to $X_{F'}$ for some finite, separable extension $F'/F$. Again taking the pushforward and dividing by the degree gives the result.
\end{proof}

\section{A counter-example}

A natural question to ask is whether or not the analogue of Corollary \ref{mm1} or Corollary \ref{mm2} holds with $\mathrm{Pic}(-)_{\Q}$ replaced by $\mathrm{Pic}(-)_{\Q_p}$. We will show in the section that when $k$ is a finite field this cannot be the case, since it would imply Tate's isogeny theorem for elliptic curves over $k\pow{t}$. Let us return to the previous notation of writing $F=k\lser{t}$ and $R=k\pow{t}$ for its ring of integers.

We first need to quickly recall some material on Dieudonn\'e modules of abelian varieties over $k,R$ and $F$. As before, we will let $W$ denote the ring of Witt vectors of $k$, set $\Omega=W\pow{t}$ and let $\Gamma$ be the $p$-adic completion of $\Omega[t^{-1}]$, so that we have $\mathcal{E}^+=\Omega[1/p]$ and $\cur{E}=\Gamma[1/p]$. Fix compatible lifts $\sigma$ of absolute Frobenius to $W\subset \Omega\subset \Gamma$. By \cite[Main Theorem 1]{dJ95b} there are covariant equivalences of categories
\[ \mathbf{D}:\mathbf{BT}_k \isomto \mathbf{DM}_W,\;\;  \mathbf{D}:\mathbf{BT}_R \isomto \mathbf{DM}_\Omega,\;\;   \mathbf{D}:\mathbf{BT}_F \isomto \mathbf{DM}_\Gamma \]
between $p$-divisible groups over $k$ (resp. $R$, $F$) and finite free Dieudonn\'e modules over $W$ (resp. $\Omega$, $\Gamma$). In particular, if $\cur{A}$ is an abelian variety over any of these rings, we will let $\mathbf{D}(\cur A)$ denote the (covariant) Dieudonn\'e module of its $p$-divisible group $\cur A[p^\infty]$. It follows essentially from the construction (see \cite{BBM82}) together with the comparison between crystalline and rigid cohomology that when $A/F$ is an abelian variety we have $\mathbf{D}(A)\otimes_\Gamma \cur{E} \cong H^1_\rig(A/\cur{E})^\vee(-1)$ as $\pn$-modules over $\cur{E}$, and from \cite[Theorem 7.0.1]{Ked00} that $\mathbf{D}(A) \otimes_\Gamma\cur{E}$ canonically descends to a $\pn$-module $\mathbf{D}^\dagger(A)\cong H^1_\rig(A/\cur{E}^\dagger)^\vee(-1)$ over $\cur{E}^\dagger$. The results of \cite[\S5.1]{BBM82} give a canonical isomorphism $\mathbf{D}^\dagger(A^\vee) \cong \mathbf{D}^\dagger(A)^\vee(-1)$ of $\pn$-modules over $\cur{E}^\dagger$. In particular, if $E$ is an elliptic curve then we have a canonical isomorphism $E\cong  E^\vee$ and hence an isomorphism $\mathbf{D}^\dagger(E)\cong \mathbf{D}^\dagger(E)^\vee(-1)$.

We can now proceed to the construction of our counter-example. It will be a smooth projective relative surface $\cur X$ over $R$, obtained as a product $\cur{E}_1 \times_R \cur{E}^\vee_2$ ($=\cur{E}_1 \times_R \cur{E}_2$) where $\cur{E}_i$ are elliptic curves over $R$ (to be specified later on). Let $X$ denote the generic fibre of $\cur X$ and $X_0$ the special fibre. As a product of elliptic curves, we know that the Tate conjecture for divisors holds for $X_0$, that is, the map
\[ c_1:\mathrm{Pic}(X_0)_{\Q_p} \rightarrow H^2_\rig(X_0/K)^{\varphi=p}  \]
is surjective. Functoriality of Dieudonn\'e modules induces a homomorphism
\[ \mathbf{D}^\dagger_{E_1,E_2}:\mathrm{Hom}(E_1,E_2)\otimes_{\Z} \Q_p \rightarrow \mathrm{Hom}_{\underline{\mathbf{M}\Phi}_{\cur{E}^\dagger}^\nabla}(\mathbf{D}^\dagger(E_1),\mathbf{D}^\dagger(E_2)) \]
which is injective by standard results.

\begin{theorem} \label{tateiso} Assume that any $\cur L\in \mathrm{Pic}(X_0)_{\Q_p}$ whose first Chern class $c_1(\cur{L})\in H^2_\rig(X/\cur{R})$ lies in the subspace $H^2_\rig(X/\cur{E}^\dagger) \subset H^2_\rig(X/\cur{R})$ lifts to $\mathrm{Pic}(\cur{X})_{\Q_p}$, in other words, assume that the $\Q_p$-analogue of Corollary \ref{mm1} holds. Then the map $\mathbf{D}^\dagger_{E_1,E_2}$ is an isomorphism.
\end{theorem}

\begin{proof}
This is essentially well-known. To start with, we note that we have a commutative diagram
\[ \xymatrix{ \mathrm{Pic}(X)_{\Q_p} \ar[r]^-{c_1}\ar@{^(->}[d] & H^2_\rig(X/\cur{E}^\dagger)^{\nabla=0,\varphi=p} \ar@{^(->}[d] \\ \mathrm{Pic}(X_0)_{\Q} \ar[r]^-{c_1} & H^2_\rig(X_0/K)^{\varphi=p}  } \]
with bottom horizontal map surjective. Under the given assumptions the top horizontal map is also surjective, and induces an isomorphism $\mathrm{NS}(X)_{\Q_p} \isomto H^2_\rig(X/\cur{E}^\dagger)^{\nabla=0,\varphi=p}$. It follows from the K\"{u}nneth formula \cite[Corollary 3.78]{LP16} that
\[ H^2_\rig(X/\cur{E}^\dagger) \cong \cur{E}^\dagger(-1) \oplus H^1_\rig(E_1/\cur{E}^\dagger) \otimes H^1_\rig(E^\vee_2/\cur{E}^\dagger) \oplus \cur{E}^\dagger(-1) \]
where the terms on either end are $H^0\otimes H^2$ and $H^2\otimes H^0$ respectively. Since $H^1_\rig(E_1/\cur{E}^\dagger)\cong \mathbf{D}^\dagger(E_1)$ and $H^1_\rig(E_2^\vee/\cur{E}^\dagger)\cong \mathbf{D}^\dagger(E_2)^\vee(-1)$ we have that
\begin{align*} H^2_\rig(X/\cur{E}^\dagger)^{\nabla=0,\varphi=p} &= \Q_p \oplus \left( \mathbf{D}^\dagger(E_1)\otimes_{\cur{E}^\dagger} \mathbf{D}^\dagger(E_2)^\vee \right)^{\nabla=0,\varphi=\mathrm{id}} \oplus  \Q_p\\
&= \Q_p \oplus \mathrm{Hom}_{\underline{\mathbf{M}\Phi}^\nabla_{\cur{E}^\dagger}}(\mathbf{D}^\dagger(E_1),\mathbf{D}^\dagger(E_2)) \oplus  \Q_p.
 \end{align*}
Next, let
$\mathrm{DC}_\mathrm{alg}(E_1,E_2^\vee)$ denote the group of divisorial correspondences from $E_1$ to $E_2^\vee$ modulo algebraic equivalence, in other words line bundles on $E_1\times E_2^\vee$ whose restriction to both $E_1\times \{0\}$ and $\{0\}\times E_2^\vee$ is trivial. Then we have shown that the map
\[ \mathrm{DC}_\mathrm{alg}(E_1,E_2^\vee)_{\Q_p} \rightarrow \mathrm{Hom}_{\underline{\mathbf{M}\Phi}_{\cur{E}^\dagger}^\nabla}(\mathbf{D}^\dagger(E_1),\mathbf{D}^\dagger(E_2)) \]
is an isomorphism, and since $\mathrm{DC}_\mathrm{alg}(E_1,E_2^\vee)_{\Q} \cong \mathrm{Hom}(E_1,E_2)_{\Q}$, it follows that the map
\[ \mathrm{Hom}(E_1,E_2)_{\Q_p} \rightarrow \mathrm{Hom}_{\underline{\mathbf{M}\Phi}_{\cur{E}^\dagger}^\nabla}(\mathbf{D}^\dagger(E_1),\mathbf{D}^\dagger(E_2)) \]
is also an isomorphism. This completes the proof. \end{proof}

In other words, to produce our required counter-example $\cur X$ we need to produce elliptic curves $\cur{E}_1$ and $\cur{E}_2$ as above such that $\mathbf{D}^\dagger_{E_1,E_2}$ is not surjective. So let $k=\F_{p^2}$ and let $E_0/k$ be a supersingular elliptic curve such that $\mathrm{Frob}_{p^2}=[p]\in \mathrm{End}(E_0)$ (such elliptic curves exist by Honda--Tate theory). It easily follows that any $\bar k$-endomorphism of $E_0$ has to commute with $\mathrm{Frob}_{p^2}$, and is hence defined over $k$. By the
$p$-adic version of Tate's isogeny theorem the $p$-divisible group functor induces an isomorphism:
$$\textrm{\rm End}(E_0)\otimes\mathbb Z_p\longrightarrow
\textrm{\rm End}(E_0[p^{\infty}]).$$

\begin{lemma}\label{many_iso} There is an isomorphism
$\phi:E_0[p^{\infty}]\to E_0[p^{\infty}]$ whose $\mathbb Q_p$-linear span in $\textrm{\rm End}(E_0[p^{\infty}])
\otimes_{\mathbb Z_p}\mathbb Q_p$ cannot be spanned by an element in
$$\textrm{\rm End}(E_0)\otimes\mathbb Q\subset
\textrm{\rm End}(E_0)\otimes\mathbb Q_p=
\textrm{\rm End}(E_0[p^{\infty}])
\otimes_{\mathbb Z_p}\mathbb Q_p.$$
\end{lemma}

\begin{proof} Since $\textrm{\rm End}(E_0[p^{\infty}])$ is an order in a quaternion algebra over $\mathbb Q_p$ by \cite[Ch. V, Theorem 3.1]{Sil86}, so its group of invertible elements is a $p$-adic Lie group of dimension at least $3$. Therefore the $\mathbb Q_p$-linear spans of elements of $\textrm{\rm End}(E_0[p^{\infty}])^*$ is uncountable. As
$\textrm{\rm End}(E_0)\otimes\mathbb Q$ is countable, there is a
$\phi\in\textrm{\rm End}(E_0[p^{\infty}])^*$ whose $\mathbb Q_p$-linear span cannot be spanned by the left hand side of the inclusion above. 
\end{proof}

Let $\cur E_1$ be an elliptic curve over $R$ whose special fibre is $E_0$ and whose generic fibre $E_1$ over $F=k\lser{t}$ is ordinary. Via the isomorphism $\phi$ in the lemma above we can consider $\cur E_1[p^{\infty}]$ as a deformation of
$E_0[p^{\infty}]$. By the Serre--Tate theorem \cite[V. Theorem 2.3]{Mes72} there is a deformation
$\cur E_2$ of $E_0$ over $R$ corresponding to this deformation of $p$-divisible groups. Let $E_2$ denote the generic fibre of $\cur E_2$ over $F$.

\begin{proposition}\label{supersingular} The map
$$\mathbf{D}^{\dagger}_{E_1,E_2}:\textrm{\rm Hom}(E_1,E_2)\otimes \mathbb Q_p\longrightarrow \mathrm{Hom}_{\underline{\mathbf{M}\Phi}_{\cur{E}^\dagger}^\nabla}(\mathbf{D}^\dagger(E_1),\mathbf{D}^\dagger(E_2)) $$
is not surjective. 
\end{proposition}

\begin{proof}
Assume for contradiction that in fact $\mathbf{D}^{\dagger}_{E_1,E_2}$ is an isomorphism. By construction
$\cur E_1[p^{\infty}]\cong
\cur E_2[p^{\infty}]$, so by the functoriality of Dieudonn\'e modules $\textrm{\rm Hom}(\mathbf{D}(\cur E_1),\mathbf D(\cur E_2))$ is non-zero. Hence $\textrm{\rm Hom}(\mathbf{D}( E_1),\mathbf D( E_2))$ is also non-zero. As
$$\mathbf{D}^{\dagger}(E_i)\otimes_{\cur{E}^\dagger} \mathcal E=\mathbf D( E_i)\otimes_\Gamma \mathcal E,$$
we get that
$\textrm{\rm Hom}(\mathbf D^{\dagger}(E_1),
\mathbf D^{\dagger}( E_2))$ is also non-zero, by Kedlaya's full faithfullness theorem \cite[Theorem 5.1]{Ked04c}. So by our assumptions $\textrm{\rm Hom}( E_1, E_2)$ is also non-zero, and the elliptic curves $ E_1$ and $ E_2$ are isogeneous. 

As $\cur E_1$ is generically ordinary but has a supersingular special fibre, it is not constant, that is, the $j$-invariant of its generic fibre $j(E_1)\not\in\overline{\mathbb F}_p$. Therefore End$( E_1)=\mathbb Z$, so by the above
$\textrm{\rm Hom}(E_1, E_2)
\otimes\mathbb Q_p$ is one-dimensional. Therefore the same holds for $\textrm{\rm Hom}(\mathbf D^{\dagger}( E_1),
\mathbf D^{\dagger}( E_2))$, too. We have a commutative diagram:
$$\xymatrix{  \textrm{\rm Hom}(\cur E_1,\cur E_2)
\otimes\mathbb Q_p\ar[r] \ar[d] &
\textrm{\rm Hom}( E_1,E_2)
\otimes\mathbb Q_p
 \ar[d]\\
\textrm{\rm Hom}(\mathbf D(\cur E_1),
\mathbf D(\cur E_2))\otimes_{\mathbb Z_p}
\mathbb Q_p\ar[r] &
\textrm{\rm Hom}(\mathbf D^{\dagger}( E_1),
\mathbf D^{\dagger}(E_2)).}$$
The lower horizontal map is an isomorphism by de Jong's full faithfullness theorem \cite{dJ98}, the upper horizontal map is an isomorphism since any abelian scheme is the N\'eron model of its generic fibre, and the right vertical map is an isomorphism by assumption. So the left vertical map is an isomorphism, too. Specialisation furnishes us with another commutative diagram:
$$\xymatrix{ \textrm{\rm Hom}(\cur E_1,\cur E_2)
\otimes\mathbb Q_p
\ar[r] \ar[d]& \textrm{\rm End}(E_0)
\otimes\mathbb Q_p\ar[d] \\
\textrm{\rm Hom}(\mathbf D(\cur E_1),
\mathbf D(\cur E_2))\otimes_{\mathbb Z_p}
\mathbb Q_p\ar[r] &
\textrm{\rm End}(\mathbf D(E_0))\otimes_{\mathbb Z_p}
\mathbb Q_p.}$$
By construction the image of the lower horizontal map in 
$$\textrm{\rm End}(\mathbf D(E_0))\otimes_{\mathbb Z_p}
\mathbb Q_p=\textrm{\rm End}(E_0[p^{\infty}])\otimes_{\mathbb Z_p}
\mathbb Q_p$$
contains the span of $\phi$. Since the domain of this map is one-dimensional, we get that its image is the span of $\phi$. Since the left vertical map is an isomorphism by the above, we get that the span of $\phi$ is spanned by the specialisation of any non-zero isogeny $\cur E_1\to\cur E_2$. This is a contradiction.
\end{proof}

We therefore arrive at the following.

\begin{corollary} There exists a smooth, projective relative surface $\cur{X}/R$ with generic fibre $X$ and special fibre $X_0$, and a class $\cur L \in \mathrm{Pic}(X_0)_{\Q_p}$ whose Chern class $c_1(\cur L)\in H^2_\rig(X/\cur{R})$ lies inside $H^2_\rig(X/\cur{E}^\dagger)$ but which does not lift to $\mathrm{Pic}(\cur X)_{\Q_p}$.
\end{corollary}

\bibliographystyle{../../Templates/bibsty}
\bibliography{../../lib.bib}

\end{document}